\documentclass[english,british]{amsart}
\usepackage[T1]{fontenc}
\usepackage[latin9]{inputenc}
\usepackage[a4paper]{geometry}
\usepackage{amsthm}
\usepackage{amssymb}
\usepackage{setspace}
\makeatletter

\numberwithin{equation}{section}
\numberwithin{figure}{section}
  \theoremstyle{plain}
  \newtheorem*{thm*}{\protect\theoremname}
  \theoremstyle{remark}
  \newtheorem*{acknowledgement*}{\protect\acknowledgementname}
\theoremstyle{plain}
\newtheorem{thm}{\protect\theoremname}[section]
  \theoremstyle{plain}
  \newtheorem{lem}[thm]{\protect\lemmaname}
  \theoremstyle{plain}
  \newtheorem{prop}[thm]{\protect\propositionname}
  \theoremstyle{remark}
  \newtheorem*{rem*}{\protect\remarkname}
  \theoremstyle{definition}
  \newtheorem{problem}[thm]{\protect\problemname}

\usepackage{amscd}
\usepackage{mathptmx}
\usepackage{bbm}

\newcommand{\e}{\mathrm e}

\newcommand{\N}{\mathbb{N}}

\newcommand{\Z}{\mathbb{Z}}

\newcommand{\R}{\mathbb{R}}

\renewcommand{\Pi}{\pi}

\renewcommand{\hat}{\widehat}

\DeclareMathOperator*{\card}{card}

\usepackage{hyperref}

\@ifundefined{textmu}
 {\usepackage{textcomp}}{}

\makeatother

\usepackage{babel}
  \addto\captionsbritish{\renewcommand{\acknowledgementname}{Acknowledgement}}
  \addto\captionsbritish{\renewcommand{\lemmaname}{Lemma}}
  \addto\captionsbritish{\renewcommand{\problemname}{Problem}}
  \addto\captionsbritish{\renewcommand{\propositionname}{Proposition}}
  \addto\captionsbritish{\renewcommand{\remarkname}{Remark}}
  \addto\captionsbritish{\renewcommand{\theoremname}{Theorem}}
  \addto\captionsenglish{\renewcommand{\acknowledgementname}{Acknowledgement}}
  \addto\captionsenglish{\renewcommand{\lemmaname}{Lemma}}
  \addto\captionsenglish{\renewcommand{\problemname}{Problem}}
  \addto\captionsenglish{\renewcommand{\propositionname}{Proposition}}
  \addto\captionsenglish{\renewcommand{\remarkname}{Remark}}
  \addto\captionsenglish{\renewcommand{\theoremname}{Theorem}}
  \providecommand{\acknowledgementname}{Acknowledgement}
  \providecommand{\lemmaname}{Lemma}
  \providecommand{\problemname}{Problem}
  \providecommand{\propositionname}{Proposition}
  \providecommand{\remarkname}{Remark}
  \providecommand{\theoremname}{Theorem}
\providecommand{\theoremname}{Theorem}

\begin{document}

\title{A Lower Bound for the Exponent of Convergence of Normal Subgroups
of Kleinian Groups}

\author{Johannes Jaerisch}

\thanks{The author was supported by the research fellowship JA 2145/1-1 of
the German Research Foundation (DFG)}

\address{Department of Mathematics, Graduate School of Science Osaka University,
1-1 Machikaneyama Toyonaka, Osaka, 560-0043 Japan }

\email{jaerisch@cr.math.sci.osaka-u.ac.jp}


\begin{abstract}
We give a short new proof that for each non-elementary
Kleinian group $\Gamma$, the exponent of convergence of an arbitrary
non-trivial normal subgroup is bounded below by half of the exponent
of convergence of $\Gamma$, and that strict inequality holds if $\Gamma$
is of divergence type. 
\end{abstract}
\maketitle

\section{Introduction and Statement of Results}

An $\left(n+1\right)$-dimensional hyperbolic  manifold, for $n\in\N$,  can be described by
the Poincar\'{e} disc model $\mathbb{D}:=\left\{ z\in\R^{n+1}:\Vert z\Vert<1\right\} $
of hyperbolic $\left(n+1\right)$-space quotiented by the 
action of a Kleinian group $\Gamma$, where we recall that a Kleinian group is a discrete subgroup
of the group of all isometries of $\mathbb{D}$ with respect to the
hyperbolic metric $d$. For each Kleinian group $\Gamma$, we can canonically associate the Poincar\'{e}
series $P\left(\Gamma,s\right):=\sum_{\gamma\in\Gamma}e^{-sd\left(0,\gamma\left(0\right)\right)}$
for each $s\in\R$. The exponent of convergence\emph{
}$\delta\left(\Gamma\right)$ of a Kleinian group $\Gamma$ is given
by the abscissa of convergence of the associated Poincar\'{e} series,
that is $\delta\left(\Gamma\right):=\inf\left\{ s\ge0:P\left(\Gamma,s\right)<\infty\right\} $. Furthermore, 
$\Gamma$ is said to be of divergence type if the series $P\left(\Gamma,\delta\left(\Gamma\right)\right)$
diverges. The quantity $\delta\left(\Gamma\right)$ carries information
about the complexity of the action of $\Gamma$ at the boundary at
infinity $\mathbb{S}:=\left\{ z\in\R^{n+1}:\Vert z\Vert=1\right\} $
of hyperbolic space. For instance, by a theorem of Bishop and Jones
(\cite{MR1484767}), it is known that $\delta\left(\Gamma\right)$
is equal to the Hausdorff dimension of the radial limit set of $\Gamma$. For more details
on Kleinian groups, limit sets and the associated hyperbolic manifolds
we refer to \cite{MR698777,MR959135,MR1041575,MR1638795,MR2191250}.

It is of interest to study how the exponents of convergence
$\delta(\hat{\Gamma})$ and $\delta\left(\Gamma\right)$ are related
for a non-trivial normal subgroup $\hat{\Gamma}$ of a Kleinian group
$\Gamma$. Each normal subgroup $\hat{\Gamma}$ corresponds to a hyperbolic
manifold $\hat{M}=\mathbb{D}/\hat{\Gamma}$ which is a normal covering
of $M=\mathbb{D}/\Gamma$. In \cite{MR783536} it is shown that if
$\Gamma$ is convex cocompact and $\delta\left(\Gamma\right)>n/2$,
then $\delta(\hat{\Gamma})=\delta\left(\Gamma\right)$ if and only
if the quotient group $\Gamma/\hat{\Gamma}$ is amenable. Recently, this result was extended to normal subgroups of essentially free Kleinian groups $\Gamma$ 
with arbitrary exponent of convergence $\delta\left(\Gamma\right)$ in \cite{Stadlbauer11}. For a recent account of the interplay between the
exponent of convergence, the Hausdorff dimension of the limit set
and the convex core entropy of Kleinian groups we refer to \cite{FalkMatsuzaki11}.

In this note we give a short new proof of the following theorem which
states that the exponent of convergence $\delta(\hat{\Gamma})$ of
a non-trivial normal subgroup $\hat{\Gamma}$ of a non-elementary
Kleinian group $\Gamma$ is bounded below by $\delta\left(\Gamma\right)/2$. This theorem
complements the results on the coincidence of $\delta(\hat{\Gamma})$
and $\delta\left(\Gamma\right)$. Before stating the result, recall that a Kleinian group is
called non-elementary if its limit set consists of more than two (and hence uncountably many) elements.
\begin{thm}\label{mainthm} 
Let $\hat{\Gamma}$ be a non-trivial normal subgroup of a non-elementary
Kleinian group $\Gamma$. We then have $\delta(\hat{\Gamma})\ge\delta\left(\Gamma\right)/2$.
If $\Gamma$ is of divergence type then we have that $\delta(\hat{\Gamma})>\delta\left(\Gamma\right)/2$.
\end{thm}
The first assertion of Theorem \ref{mainthm} was  obtained by Falk and Stratmann (\cite[Theorem 2]{MR2097162})
using a result of Matsuzaki (\cite[Theorem 6]{MR2038133}). The complete statement of  Theorem \ref{mainthm} was  proved by Roblin (\cite{MR2166367}). 

The novelty in our proof of Theorem \ref{mainthm} is to show that there exists a uniformly finite-to-one map from a factor of the group $\Gamma$ to the normal subgroup $\hat{\Gamma}$ (see Lemma \ref{lem:finite-to-one}). Combining this with elementary
hyperbolic geometry,  the first assertion of  Theorem  \ref{mainthm} follows. For
the second assertion, we additionally make use of Proposition
\ref{prop:divergencetype-same-exponent}, which is taken from \cite[Corollary 4.3]{MR2486788}
and which is based on a uniqueness property of the Patterson-Sullivan
measure for Kleinian groups of divergence type. 

We remark that in certain special cases we are able to simplify the
proof of Theorem \ref{mainthm} even further by replacing the uniformly finite-to-one
map (see the definition prior to Lemma \ref{lem:finite-to-one}) by
a one-to-one map from $\Gamma$ to $\hat{\Gamma}$ (see Proposition
\ref{prop:one-to-one-using-malnormalsubgroups}). These special cases
include the case where $\Gamma$ is a free group of rank greater than
one and $\hat{\Gamma}$ is an arbitrary non-trivial normal subgroup
of $\Gamma$, as well as all the non-elementary Kleinian groups $\Gamma$
and non-trivial normal subgroups $\hat{\Gamma}$ such that $\hat{\Gamma}$
contains a free subgroup of rank two which is a malnormal subgroup
of $\Gamma$. We remark that regarding Fuchsian groups ($n=1$), all
torsion-free Fuchsian groups not corresponding to a closed surface
are free groups and are therefore covered by Proposition \ref{prop:one-to-one-using-malnormalsubgroups}.
Furthermore, by a result of \cite{MR1713122}, it is known that any
non-elementary subgroup of a torsion-free  hyperbolic group $\Gamma$
contains a free group of rank two which is malnormal in $\Gamma$.
Thus, the second special case of Proposition \ref{prop:one-to-one-using-malnormalsubgroups}
also covers all convex cocompact Kleinian groups $\Gamma$. Using
the concept of relative hyperbolicity it is possible to apply Proposition
\ref{prop:one-to-one-using-malnormalsubgroups} to all non-elementary
geometrically finite groups $\Gamma$. In the case of Fuchsian groups,
we have that geometrically finite groups coincide with finitely generated
groups, whereas in higher dimension we only have that geometrically
finite groups are finitely generated. However, in the case $n=2$,
we can use that for each finitely generated group $\Gamma$ there
exists a geometrically finite group $\Gamma'$ which is isomorphic
to $\Gamma$. Thus, Proposition \ref{prop:one-to-one-using-malnormalsubgroups}
is applicable to $\Gamma'$ and all its non-trivial normal subgroups.
Since the result of Proposition \ref{prop:one-to-one-using-malnormalsubgroups}
is in the setting of abstract groups, it also applies to all non-trivial
normal subgroups of all non-elementary finitely generated Kleinian
groups in dimension $n=2$. It would be interesting to know if Proposition
\ref{prop:one-to-one-using-malnormalsubgroups} can be extended to
all non-elementary Kleinian groups (see Problem \ref{problem}).

Recently, the results stated in Theorem \ref{mainthm} were partially obtained in \cite{MatsuTaylorTaylor11} in the special case that $\Gamma$ is convex cocompact. The proof given in there uses ergodicity of the geodesic flow on $\mathbb{D} / \Gamma$ and makes use of a result of Lundh (\cite{MR1992947}). In \cite{MatsuTaylorTaylor11},   it is also shown that the lower bound 
in Theorem \ref{mainthm} is sharp in the sense that for certain groups of the
first kind $\Gamma$,  there exists a sequence of normal subgroups $(\hat{\Gamma}_{n})_{n\in \N}$
of $\Gamma$ such that $\lim_{n\rightarrow \infty}\delta(\hat{\Gamma}_{n})=\delta\left(\Gamma\right)/2$. 

The existence  of a finite-to-one map from a factor of $\Gamma$ to the normal subgroup $\hat{\Gamma}$, which is employed in this paper, was also  used in the author's doctoral dissertation (\cite[Theorem 6.2.10]{JaerischDissertation11})
in order to obtain the first assertion of Theorem \ref{mainthm} for normal subgroups
of Fuchsian groups of Schottky type. \foreignlanguage{english}{Furthermore,
this idea was used to obtain the analogue of  Theorem  \ref{mainthm} in the context of} fractal models of normal subgroups
of Schottky groups in \cite[Theorem 1.2]{Jaerisch11}. Recently,  the author has extended Proposition \ref{prop:divergencetype-same-exponent}, which was vital in the proof of the second assertion of Theorem \ref{mainthm},  in terms of the thermodynamic formalism  of group-extended Markov systems (\cite{Jaerisch12b}).

Let us end this introduction with a brief discussion of further generalisations of Theorem \ref{mainthm} to    isometry groups of    $\mathrm{CAT}(-1)$ spaces and Gromov hyperbolic groups.   Matsuzaki and Yabuki have informed the author about  work in progress which  extends  Proposition \ref{prop:divergencetype-same-exponent} to isometry groups  acting properly discontinuously on a  proper $\mathrm{CAT}(-1)$ space, and  quasiconvex cocompact hyperbolic groups.  Provided that an extension of Proposition \ref{prop:divergencetype-same-exponent} holds, and that Proposition \ref{prop:one-to-one-using-malnormalsubgroups} applies, our proof of Theorem \ref{mainthm} can be  verbatim extended  to  isometry groups of a   $\mathrm{CAT}(-1)$ space and Gromov hyperbolic groups. As  an   example, consider the free group $\Gamma$ with at least two generators, which acts isometrically on its Cayley graph $(X,d)$ with respect to the word metric $d$. Then $(X,d)$ is a $\mathrm{CAT}(-1)$ space and Theorem \ref{mainthm} holds for each non-trivial normal subgroup of $\Gamma$. It would be interesting to know to what extent  Proposition \ref{prop:divergencetype-same-exponent} can be further generalised. 

\begin{acknowledgement*}
The author would like to thank Professor Matsuzaki for inviting him
to \linebreak Waseda University and for the warm hospitality during
his stay in Tokyo, where this paper was finished. The author is grateful
for fruitful discussion with Professor Matsuzaki on the subject of
this paper, bringing the relative Poincar\'{e} series and the concept
of malnormal subgroups to the author's attention.  The author thanks an anonymous referee for his valuable suggestions on an earlier version of this paper.
\end{acknowledgement*}

\section{Proof of the Theorem}

Let us begin by proving an elementary lemma which allows us to
investigate the Poincar\'{e} series in terms of a certain relative
Poincar\'{e} series. For the hyperbolic metric $d$ and $A\subset\mathbb{D}$
we set $d\left(0,A\right):=\inf_{x\in A}d\left(0,x\right)$. For a
subgroup $H$ of a group $G$ we denote by $H\backslash G$ the set
of right cosets $\left\{ Hg:g\in G\right\} $ and for each $g\in G$
we set $\left[g\right]:=Hg$.
\begin{lem}
\label{lem:relativepoincare}Let $\Gamma$ be a non-elementary Kleinian
group. For each hyperbolic element $h\in\Gamma$ and for each $s>0$
there exists a constant $C>0$ such that
\[
\sum_{\gamma\in\Gamma}\e^{-sd\left(0,\gamma\left(0\right)\right)}\le C\sum_{\left[g\right]\in\left\langle h\right\rangle \backslash\Gamma}\e^{-sd\left(0,\left[g\right]\left(0\right)\right)}.
\]
\end{lem}
\begin{proof}
Without loss of generality we may assume that the origin is an element
of the axis $C\left(h\right)$ of the hyperbolic element $h$ which
joins the fixed points $p$ and $q$ of $h$. For each $\left[g\right]\in\left\langle h\right\rangle \backslash\Gamma$
we choose $g_{0}\in\left[g\right]$ such that the orthogonal projection
$P_{0}$ of $g_{0}\left(0\right)$ on $C\left(h\right)$ satisfies
$d\left(0,P_{0}\right)\le d\left(0,h\left(0\right)\right)$. We then
clearly have that
\[
d\left(0,\left[g\right]\left(0\right)\right)\le d\left(0,g_{0}\left(0\right)\right)\le d\left(0,h\left(0\right)\right)+d\left(P_{0},g_{0}\left(0\right)\right).
\]
Using this estimate, it follows for each $n\in\Z$ that the orthogonal
projection $P_{n}$ of $h^{n}g_{0}\left(0\right)$ on $C\left(h\right)$,
which is given by $P_{n}=h^{n}P_{0}$, satisfies
\begin{equation}
d\left(P_{n},h^{n}g_{0}\left(0\right)\right)=d\left(P_{0},g_{0}\left(0\right)\right)\ge d\left(0,\left[g\right]\left(0\right)\right)-d\left(0,h\left(0\right)\right).\label{eq:relative-poincare-1}
\end{equation}
By the First Law of Cosines (\cite[Theorem 3.5.3]{MR2249478} for
right-angled hyperbolic triangles we can estimate for each $n\in\Z$
that
\begin{equation}
d\left(0,h^{n}g_{0}\left(0\right)\right)+2\log2\ge d\left(0,P_{n}\right)+d\left(P_{n},h^{n}g_{0}\left(0\right)\right).\label{eq:relative-poincare-2}
\end{equation}
Combining the previous inequality with the fact that $d\left(0,P_{n}\right)\ge\left(\left|n\right|-1\right)d\left(0,h\left(0\right)\right)$
and using the inequality in (\ref{eq:relative-poincare-1}) we conclude
for each $n\in\Z$ that
\begin{eqnarray*}
d\left(0,h^{n}g_{0}\left(0\right)\right)+2\log2 & \ge & \left(\left|n\right|-1\right)d\left(0,h\left(0\right)\right)+d\left(0,\left[g\right]\left(0\right)\right)-d\left(0,h\left(0\right)\right)\\
 & = & \left(\left|n\right|-2\right)d\left(0,h\left(0\right)\right)+d\left(0,\left[g\right]\left(0\right)\right).
\end{eqnarray*}
Consequently, we obtain for each $s>0$ and for all $n\in\Z$ that
\[
\e^{-sd\left(0,h^{n}g_{0}\left(0\right)\right)}\le2^{-2s}\e^{-\left(\left|n\right|-2\right)d\left(0,h\left(0\right)\right)}\e^{-sd\left(0,\left[g\right]\left(0\right)\right)}.
\]
Summing the previous inequality over $n\in\Z$ shows that there exists
a constant $C>0$ depending on $s>0$ and $h\in\Gamma$ such that
for each \foreignlanguage{english}{$\left[g\right]\in\left\langle h\right\rangle \backslash\Gamma$
we have} \foreignlanguage{english}{
\begin{equation}
\sum_{\gamma\in\left[g\right]}\e^{-sd\left(0,\gamma\left(0\right)\right)}=\sum_{n\in\Z}\e^{-sd\left(0,h^{n}g_{0}\left(0\right)\right)}\le C\e^{-sd\left(0,\left[g\right]\left(0\right)\right)}.\label{eq:relative-poincare-3}
\end{equation}
Finally, summing over $\left[g\right]\in\left\langle h\right\rangle \backslash\Gamma$
in (\ref{eq:relative-poincare-3}) finishes the proof of the lemma.}
\end{proof}
For each normal subgroup $\hat{\Gamma}$ of $\Gamma$ and for each
$h\in\hat{\Gamma}$ we consider the map $\iota_{h}:\Gamma\rightarrow\hat{\Gamma}$
which is given by $\iota_{h}(g):= g^{-1}hg$, for each $g\in\Gamma$. Since
for each $n\in\Z$ and $g\in\Gamma$ we have $\iota_{h}\left(h^{n}g\right)=g^{-1}h^{-n}hh^{n}g=\iota_{h}\left(g\right)$,
this defines a map $\iota_{h}:\left\langle h\right\rangle \backslash\Gamma\rightarrow\hat{\Gamma}$.
The next lemma shows that this map is uniformly finite-to-one.
\begin{lem}
\label{lem:finite-to-one}Let $\hat{\Gamma}$ be a non-trivial normal
subgroup of a non-elementary Kleinian group $\Gamma$. For each hyperbolic
element $h\in\hat{\Gamma}$ there exists a constant $k\in\N$ such
that the map $\iota_{h}:\left\langle h\right\rangle \backslash\Gamma\rightarrow\hat{\Gamma}$
is at most $k$-to-one. \end{lem}
\begin{proof}
Let $p$ and $q$ denote the fixed points of $h$ and let $H$ denote
the subgroup of $\Gamma$ which preserves the fixed points of $h$,
that is, $H:=\left\{ g\in\Gamma:g\left(\left\{ p,q\right\} \right)=\left\{ p,q\right\} \right\} $.
Since the limit set of $H$ is equal to $\left\{ p,q\right\} $, we
have that $H$ is an elementary group containing the hyperbolic element
$h$. Therefore, $H$ is a finite extension of the cyclic group \foreignlanguage{english}{$\left\langle h_{0}\right\rangle $}
generated by some hyperbolic element $h_{0}\in H$,  and there exists
$l\in\Z$ such that $h=h_{0}^{l}$. We conclude that $\left\langle h\right\rangle $
is a subgroup of $H$ with finite index $\left[H:\left\langle h\right\rangle \right]=k-1$,
for some $k\ge2$.

We now show that $\iota_{h}:\left\langle h\right\rangle \backslash\Gamma\rightarrow\hat{\Gamma}$
is at most $k$-to-one. Let $\left[g_{1}\right],\dots,\left[g_{k+1}\right]\in\left\langle h\right\rangle \backslash\Gamma$
be given such that $\iota_{h}\left(\left[g_{1}\right]\right)=\dots=\iota_{h}\left(\left[g_{k+1}\right]\right)$.
Then for each $j\in\left\{ 1,\dots,k\right\} $ we have that $g_{j}^{-1}hg_{j}=g_{k+1}^{-1}hg_{k+1}$,
which implies that $hg_{j}g_{k+1}^{-1}=g_{j}g_{k+1}^{-1}h$. Therefore,
\foreignlanguage{english}{for each} $j\in\left\{ 1,\dots,k\right\} $
we have that $g_{j}g_{k+1}^{-1}$ commutes with $h$, which implies
that the fixed points of $h$ are preserved by \foreignlanguage{english}{$g_{j}g_{k+1}^{-1}$.
Consequently, for each }$j\in\left\{ 1,\dots,k\right\} $\foreignlanguage{english}{
we deduce that }$g_{j}g_{k+1}^{-1}\in H$\foreignlanguage{english}{.
Since $\left[H:\left\langle h\right\rangle \right]=k-1$ it follows
by the }pigeonhole principle that there exist distinct integers $m,n\in\left\{ 1,\dots,k\right\} $
such that $\left[g_{m}g_{k+1}^{-1}\right]=\left[g_{n}g_{k+1}^{-1}\right]$
in $\left\langle h\right\rangle \backslash H$ and hence, $\left[g_{m}\right]=\left[g_{n}\right]$
in $\left\langle h\right\rangle \backslash\Gamma$. The proof is complete.
\end{proof}
For the sake of completeness we cite the following result from \cite[Corollary 4.3]{MR2486788}.
\begin{prop}
\label{prop:divergencetype-same-exponent}Let $\hat{\Gamma}$ denote
a Kleinian group of divergence type and let $\Gamma$ be a Kleinian
group which contains $\hat{\Gamma}$ as a normal subgroup. We then
have that $\delta(\hat{\Gamma})=\delta\left(\Gamma\right)$.\end{prop}

We are now in the position to prove the main theorem.
\begin{proof}
[Proof of Theorem \ref{mainthm}]Since $\Gamma$ is non-elementary and $\hat{\Gamma}$
is a non-trivial normal subgroup of $\Gamma$, it follows that the limit sets of $\Gamma$ and $\hat{\Gamma}$
coincide. Hence, $\hat{\Gamma}$ is also non-elementary
(see e.g. \cite[Lemma 2.2]{MR1638795}). Fix a hyperbolic element
$h\in\hat{\Gamma}$ and let $\iota_{h}:\left\langle h\right\rangle \backslash\Gamma\rightarrow\hat{\Gamma}$
denote the map defined prior to Lemma \ref{lem:finite-to-one}, which
is at most $k$-to-one by Lemma \ref{lem:finite-to-one}, for some
$k\in\N$. By Lemma \ref{lem:relativepoincare} we have that,  for each $s>0$,  there exists a constant
$C(s)>0$ such that
\begin{equation}
\sum_{\gamma\in\Gamma}\e^{-sd\left(0,\gamma\left(0\right)\right)}\le C(s)\sum_{\left[g\right]\in\left\langle h\right\rangle \backslash\Gamma}\e^{-sd\left(0,\left[g\right]\left(0\right)\right)}.\label{eq:mainproof-1}
\end{equation}
An application of the triangle inequality shows that for each $g\in\Gamma$
we have that \foreignlanguage{english}{
\begin{eqnarray}
d\left(0,g^{-1}hg\left(0\right)\right) & \le & d\left(0,g^{-1}\left(0\right)\right)+d\left(g^{-1}\left(0\right),g^{-1}h\left(0\right)\right)+d\left(g^{-1}h\left(0\right),g^{-1}hg\left(0\right)\right)\label{eq:mainproof-1a}\\
 & = & d\left(0,g^{-1}\left(0\right)\right)+d\left(0,h\left(0\right)\right)+d\left(0,g\left(0\right)\right)=2d\left(g\left(0\right),0\right)+d\left(h\left(0\right),0\right).\nonumber
\end{eqnarray}
From this we deduce that} for each $\left[g\right]\in\left\langle h\right\rangle \backslash\Gamma$
we have
\begin{equation}
\e^{-d\left(0,\left[g\right]\left(0\right)\right)}\le\e^{d\left(0,h\left(0\right)\right)/2}\e^{-d\left(0,\iota_{h}\left(\left[g\right]\right)\left(0\right)\right)/2}.\label{eq:mainproof-2}
\end{equation}
Combining first (\ref{eq:mainproof-1}) and (\ref{eq:mainproof-2})
and then using that the map $\iota_{h}:\left\langle h\right\rangle \backslash\Gamma\rightarrow\hat{\Gamma}$
is at most $k$-to-one, we conclude that for each $s>0$, 
\begin{equation}
\sum_{\gamma\in\Gamma}\e^{-sd\left(0,\gamma\left(0\right)\right)}\le C(s)\e^{sd\left(0,h\left(0\right)\right)/2}\sum_{\left[g\right]\in\left\langle h\right\rangle \backslash\Gamma}\e^{-sd\left(0,\iota_{h}\left(\left[g\right]\right)\left(0\right)\right)/2}\le kC(s)\e^{sd\left(0,h\left(0\right)\right)/2}\sum_{\rho\in\hat{\Gamma}}\e^{-sd\left(0,\rho\left(0\right)\right)/2}.\label{eq:mainproof-3}
\end{equation}
We can now prove the first assertion of Theorem  \ref{mainthm}. For each $\epsilon>0$
we have that $\sum_{\gamma\in\Gamma}\e^{-\left(\delta\left(\Gamma\right)-\epsilon\right)d\left(0,\gamma\left(0\right)\right)}=\infty$.
Hence, it follows by (\ref{eq:mainproof-3}) that $\sum_{\rho\in\hat{\Gamma}}\e^{-\left(\delta\left(\Gamma\right)-\epsilon\right)d\left(0,\rho\left(0\right)\right)/2}=\infty$.
Since this holds for each $\epsilon>0$, we conclude that $\delta(\hat{\Gamma})\ge\delta\left(\Gamma\right)/2$.
For the proof of the second assertion let $\Gamma$ be of divergence
type and suppose by way of contradiction that $\delta(\hat{\Gamma})=\delta\left(\Gamma\right)/2$.
Then it follows by (\ref{eq:mainproof-3}) that $\hat{\Gamma}$ is
of divergence type. Since $\hat{\Gamma}$ is a normal subgroup of
$\Gamma$, it follows by Proposition \ref{prop:divergencetype-same-exponent}
that $\delta(\hat{\Gamma})=\delta\left(\Gamma\right)$. But from this
we deduce that $\delta\left(\Gamma\right)=\delta\left(\Gamma\right)/2$
which implies that $\delta\left(\Gamma\right)=0$. This gives the
desired contradiction, since it is well-known that for each non-elementary
group $\Gamma$ we have that $\delta\left(\Gamma\right)>0$ (see e.g.
\cite{MR0227402}). The proof is complete.
\end{proof}

\subsection{Special cases }

In certain special cases the following proposition can further simplify
the proof of Theorem \ref{mainthm}. The next proposition considers abstract
groups and might be of independent interest. Recall that a subgroup
$H$ of a group $G$ is called a malnormal subgroup of $G$ if for
each $g\in G$ with $g\notin H$ we have that $(gHg^{-1})\cap H=\left\{ 1\right\} $.
\begin{prop}
\label{prop:one-to-one-using-malnormalsubgroups}Let $\hat{\Gamma}$
be a non-trivial normal subgroup of a group $\Gamma$. If $\Gamma$
is free of rank greater than one, or if $\hat{\Gamma}$ contains a
free subgroup $H=\left\langle h_{1},h_{2}\right\rangle $ of rank
two which is a malnormal subgroup of $\Gamma$, then there exists
a finite set $F\subset\hat{\Gamma}$ and a map $\tau:\Gamma\rightarrow F$
such that the map $\iota:\Gamma\rightarrow\hat{\Gamma}$, given by
$\iota\left(g\right):=g\tau\left(g\right)g^{-1}$, is one-to-one.
\end{prop}
Before we turn to the proof of the proposition let us explain  how it
 can be  applied in order to simplify the proof of  Theorem \ref{mainthm}. Combining the proposition with the  inequality (\ref{eq:mainproof-1a}) in the proof of 
Theorem \ref{mainthm}, we can immediately deduce that for each $s>0$, 
\[
\sum_{g\in\Gamma}\e^{-sd\left(0,g\left(0\right)\right)}\le\max_{h\in F}\e^{sd\left(0,h\left(0\right)\right)/2}\sum_{\rho\in\hat{\Gamma}}\e^{-sd\left(0,\rho\left(0\right)\right)/2},
\]
which is analogous to the inequality in (\ref{eq:mainproof-3}) above.
We can then follow the proof of Theorem \ref{mainthm} after (\ref{eq:mainproof-3})
and obtain a proof  without having to apply Lemma \ref{lem:relativepoincare}
and Lemma \ref{lem:finite-to-one}.
\begin{proof}
[Proof of Proposition \ref{prop:one-to-one-using-malnormalsubgroups}]First, 
we consider the case that the group $\Gamma=\left\langle \gamma_{i}:i\in I\right\rangle $
with $\card\left(I\right)>1$ is free. Fix some arbitrary $h_0\in \hat{\Gamma}\setminus\left\{ 1\right\} $ and set $F:=\left\{ \gamma_{1}^{\pm1}h_0 \gamma_{1}^{\mp1} ,\gamma_{2}^{\pm1}h_0 \gamma_{2}^{\mp1}\right\} $. 
Let $h_{0}$ be given by the reduced word $\gamma_{\omega_{1}}^{\epsilon_{1}}\cdot\dots\cdot\gamma_{\omega_{l}}^{\epsilon_{l}}$,
 where  $l\in\N$, $\omega\in I^{l}$ and $\epsilon\in\left\{ \pm1\right\} ^{l}$.
We  define $\tau:\Gamma\rightarrow F$ as follows. For $g\in\Gamma$, 
given by the reduced word $\gamma_{\eta_{1}}^{\rho_{1}}\cdot\dots\cdot\gamma_{\eta_{m}}^{\rho_{m}}$, 
where $m\in \N$,  $\eta\in I^{m}$ and $\rho\in\left\{ \pm1\right\} ^{m}$, we choose $\alpha \in \left\{ \gamma_{1}^{\pm1} ,\gamma_{2}^{\pm1}\right\} \setminus\left\{ \gamma_{\eta_{m}}^{-\rho_{m}},\gamma_{\omega_{1}}^{-\epsilon_{1}},\gamma_{\omega_{l}}^{\epsilon_{l}}\right\}   $ and we set 
$\tau(g):=\alpha h_0 \alpha^{-1}\in F$.  
Then, the element  $\iota\left(g\right)=g\alpha h_0 \alpha^{-1}g^{-1} \in\hat{\Gamma}$ is   given by
the reduced word $\gamma_{\eta_{1}}^{\rho_{1}}\cdot\dots\cdot\gamma_{\eta_{m}}^{\rho_{m}} \alpha  \gamma_{\omega_{1}}^{\epsilon_{1}}\dots\gamma_{\omega_{l}}^{\epsilon_{l}}\alpha^{-1}\gamma_{\eta_{m}}^{-\rho_{m}}\cdot\dots\cdot\gamma_{\eta_{1}}^{-\rho_{1}}$, 
which proves that  $\iota$ is one-to-one.

Secondly, we consider the case that $\hat{\Gamma}$ contains a free
subgroup $H=\left\langle h_{1},h_{2}\right\rangle $ of rank two which
is a malnormal subgroup of $\Gamma$. Consider the left coset decomposition
of $\Gamma$ with respect to $H$, that is $\Gamma=\bigcup_{j\in J}g_{j}H$,
for some index set $J$ and a fixed choice of $\left(g_{j}\right)\in\Gamma^{J}$.
Then,  for each $g\in\Gamma$,  there exist  $j\in J$,   $l\in\N$,  words
$\omega\in\left\{ 1,2\right\} ^{l}$ and $\epsilon\in\left\{ \pm1\right\} ^{l}$, 
such that $g=g_{j}h_{\omega_{1}}^{\epsilon_{1}}\cdot\dots\cdot h_{\omega_{l}}^{\epsilon_{l}}$
and $h_{\omega_{1}}^{\epsilon_{1}}\cdot\dots\cdot h_{\omega_{l}}^{\epsilon_{l}}$
is reduced. We then set $F:=\left\{ h_{1},h_{2}\right\} $ and define
$\tau:\Gamma\rightarrow F$ as follows. If $\omega_{l}=1$ then we
set $\tau\left(g\right):=h_{2}$, otherwise we set $\tau\left(g\right):=h_{1}$.
We now verify that  $\iota:\Gamma\rightarrow\hat{\Gamma}$
  is one-to-one. Let
$g,g'\in\Gamma$ with $\iota\left(g\right)=\iota\left(g'\right)$
be given. There exist indices $i,j\in J$, integers $l,m\in\N$ and
words $\omega\in I^{l}$, $\eta\in I^{m}$, $\epsilon\in\left\{ \pm1\right\} ^{l}$
and $\rho\in\left\{ \pm1\right\} ^{m}$ such that $g=g_{i}h$ and
$g'=g_{j}h'$ with reduced words $h=h_{\omega_{1}}^{\epsilon_{1}}\cdot\dots\cdot h_{\omega_{l}}^{\epsilon_{l}}$
and $h'=h_{\eta_{1}}^{\rho_{1}}\cdot\dots\cdot h_{\eta_{m}}^{\rho_{m}}$.
Note that $\iota\left(g\right)=\iota\left(g'\right)$ implies
that
\begin{equation}
g_{j}^{-1}g_{i}h\tau\left(g\right)h^{-1}g_{i}^{-1}g_{j}=h'\tau\left(g'\right)\left(h'\right)^{-1}.\label{eq:malnormal-1}
\end{equation}
Observe that by definition
of $\tau$,  we have that   $h\tau\left(g\right)h^{-1}$ and \foreignlanguage{english}{$h'\tau\left(g'\right)\left(h'\right)^{-1}$
are both elements of $H\setminus \{1\}$. Since $H$ is a malnormal subgroup of $\Gamma$, the equality
in (\ref{eq:malnormal-1}) implies that $g_{j}^{-1}g_{i}\in H$ and
thus $i=j$. } Now \foreignlanguage{english}{(\ref{eq:malnormal-1})}
implies that  $h=h'$, which gives  $g=g_{i}h=g_{i}h'=g'$.
The proof is complete.\end{proof}
Let us end this note with the following interesting open question.
\begin{problem}
\label{problem}Let $\hat{\Gamma}$ be a non-trivial normal subgroup
of a non-elementary Kleinian group $\Gamma$. Does there exist a finite
set $F\subset\hat{\Gamma}$ and a map $\tau:\Gamma\rightarrow F$
such that the map $\iota:\Gamma\rightarrow\hat{\Gamma}$, which is for each $g\in \Gamma$  given by
$\iota(g):=g\tau\left(g\right)g^{-1}$,  is one-to-one?
\end{problem}
\providecommand{\bysame}{\leavevmode\hbox to3em{\hrulefill}\thinspace}
\providecommand{\MR}{\relax\ifhmode\unskip\space\fi MR }
\providecommand{\MRhref}[2]{%
  \href{http://www.ams.org/mathscinet-getitem?mr=#1}{#2}
}
\providecommand{\href}[2]{#2}

\end{document}